\documentclass[12pt,reqno]{amsart}
\usepackage{amsmath, amsthm, amscd, amsfonts, amssymb, graphicx, color}
\input{mathrsfs.sty}
\newtheorem{theorem}{Theorem}[section]

\theoremstyle{definition}

\theoremstyle{remark}

\begin{document}

\title[Hermite-Hadamard's Type Inequalities on a Ball]{Hermite-Hadamard's Type Inequalities on a Ball }

\author[M. Rostamian Delavar]{M. Rostamian Delavar}
\address{Department of Mathematics, Faculty of Basic Sciences, University of Bojnord, P. O. Box 1339, Bojnord 94531, Iran}
\email{\textcolor[rgb]{0.00,0.00,0.84}{m.rostamian@ub.ac.ir}}




\begin{abstract}
Some trapezoid and mid-point type inequalities related to the Hermite-Hadamard inequality for the mappings defined on a ball in
the space are obtained.
\end{abstract}

\maketitle

\section{Introduction}
Consider the ball $\bar{B}(C,R)$ in the space where $C=(a,b,c)\in \mathbb{R}^3$, $R>0$ and $$\bar{B}(C,R)=\{(x,y,z)\in \mathbb{R}^3|(x-a)^2+(y-b)^2+(z-c)^2\leq R^2\}.$$
Also consider $\sigma(C,R)$ as the boundary of $\bar{B}(C,R)$, i.e. $$\sigma(C,R)=\{(x,y,z)\in \mathbb{R}^3|(x-a)^2+(y-b)^2+(z-c)^2= R^2\}.$$
The following result has proved in \cite{drag}, which is the Hermite-Hadamard inequality for convex functions defined on a ball $\bar{B}(C,R)$.
\begin{theorem}
Let $\bar{B}(C,R)\to \mathbb{R}$ be a convex mapping on the ball $\bar{B}(C,R)$. Then we have the inequality:
\begin{align}\label{eq.08}
&f(a,b,c)\leq\frac{3}{4\pi R^3}\int\int\int_{\bar{B}(C,R)}f(x,y,z)dv\leq\\
&\frac{1}{4\pi R^2}\int\int_{\sigma(C,R)}f(x,y,z)d\sigma.\notag
\end{align}
\end{theorem}
Motivated by (\ref{eq.08}), we obtain some trapezoid and mid-point type inequalities related to the Hermite-Hadamard inequality for the mappings defined on a ball $\bar{B}(C,R)$ in
the space. In this paper we use the spherical coordinates to prove our results.

\section{Main Results}

The following is trapezoid type inequalities related to the (\ref{eq.08}) for the mappings defined on $\bar{B}(C,R)$.
\begin{theorem}
Suppose that $\bar{B}(C,R)\subset I^\circ$, where $I\subset R^3$ and consider $f:\bar{B}(C,R)\to \mathbb{R}$ which has continuous partial derivatives with respect to the variables $\rho$, $\varphi$ and $\theta$ on $I^\circ$ in spherical coordinates. If $|\frac{\partial f}{\partial \rho}|$ is convex on $\bar{B}(C,R)$, then
\begin{align}\label{eq.08'}
&\bigg|\frac{1}{4\pi R^2}\int\int_{\sigma(C,R)}f(x,y,z)d\sigma-\frac{1}{\frac{4}{3}\pi R^3}\int\int\int_{\bar{B}(C,R)}f(x,y,z)dv\bigg|\leq\\
&\frac{1}{16\pi R}\int\int_{\sigma(C,R)}\big|\frac{\partial f}{\partial \rho}\big|(x,y,z)d\sigma\notag.
\end{align}
Furthermore above inequality is sharp.
\end{theorem}

\begin{proof}
First notice that
\begin{align}\label{eq.09}
&\int\int\int_{\bar{B}(C,R)}f(x,y,z)dv=\\
&\int_0^{2\pi}\int_0^{\pi}\int_0^R f(a+\rho cos\theta sin\varphi, b+\rho sin\theta cos\varphi, c+\rho cos\varphi)\rho^2 sin\varphi d\rho d\varphi d\theta.\notag
\end{align}

Second notice that
\begin{align}\label{eq.10}
&\int\int_{\sigma(C,R)}f(x,y,z)d\sigma=\\
&\int_0^{2\pi}\int_0^{\pi} f(a+R cos\theta sin\varphi, b+R sin\theta cos\varphi, c+R cos\varphi)R^2 sin\varphi d\varphi d\theta.\notag
\end{align}

Now for fixed $\varphi\in [0,\pi]$ and $\theta\in [0,2\pi]$, we have
\begin{align}\label{eq.11}
&\int_0^R\frac{\partial f}{\partial \rho}(a+\rho cos\theta sin\varphi, b+\rho sin\theta cos\varphi, c+\rho cos\varphi)\rho^3 sin\varphi d\rho=\\
&R^3 f(a+R cos\theta sin\varphi, b+R sin\theta cos\varphi, c+R cos\varphi)-\notag\\
&-3\int_0^R f(a+\rho cos\theta sin\varphi, b+\rho sin\theta cos\varphi, c+\rho cos\varphi)\rho^2 sin\varphi d\rho.\notag
\end{align}

So integrating with respect to $\varphi\in [0,\pi]$ and $\theta\in [0,2\pi]$ in (\ref{eq.11}) along with (\ref{eq.09}), (\ref{eq.10}) and the convexity of $\big|\frac{\partial f}{\partial \rho}\big|$ imply that
\begin{align}\label{eq.12}
&\bigg|R\int\int_{\sigma(C,R)}f(x,y,z)d\sigma-3\int\int\int_{\bar{B}(C,R)}f(x,y,z)dv\bigg|\leq\\
&\int_0^{2\pi}\int_0^{\pi}\int_0^R \big|\frac{\partial f}{\partial \rho}\big|(a+\rho cos\theta sin\varphi, b+\rho sin\theta cos\varphi, c+\rho cos\varphi)\rho^3 sin\varphi d\rho d\varphi d\theta\leq\notag\\
&\int_0^{2\pi}\int_0^{\pi}\int_0^R \big|\frac{\partial f}{\partial \rho}\big|\Big((1-\frac{\rho}{R})(a,b,c)+\frac{\rho}{R}(a+Rcos\theta sin\varphi, b+R sin\theta cos\varphi, c+R cos\varphi\Big)\times\notag\\
&\rho^3 sin\varphi d\rho d\varphi d\theta\leq\notag
\end{align}
\begin{align}
&\int_0^{2\pi}\int_0^{\pi}\int_0^R \rho^3 \big(1-\frac{\rho}{R}\big) \big|\frac{\partial f}{\partial \rho}\big|(a,b,c)sin\varphi d\rho d\varphi d\theta+\notag\\
&\int_0^{2\pi}\int_0^{\pi}\int_0^R \frac{\rho^4}{R}\big|\frac{\partial f}{\partial \rho}\big|\big(a+Rcos\theta sin\varphi, b+R sin\theta cos\varphi, c+R cos\varphi\big)sin\varphi  d\rho d\varphi d\theta=\notag\\
&\frac{\pi R^4}{5}\big|\frac{\partial f}{\partial \rho}\big|(a,b,c)+\notag\\
&\frac{R^4}{5}\int_0^{2\pi}\int_0^{\pi}\big|\frac{\partial f}{\partial \rho}\big|\big(a+Rcos\theta sin\varphi, b+R sin\theta cos\varphi, c+R cos\varphi\big)sin\varphi  d\varphi d\theta=\notag\\
&\frac{\pi R^4}{5}\big|\frac{\partial f}{\partial \rho}\big|(a,b,c)+\frac{R^2}{5}\int\int_{\sigma(C,R)}\big|\frac{\partial f}{\partial \rho}\big|(x,y,z)d\sigma\notag.
\end{align}
Since $\big|\frac{\partial f}{\partial \rho}\big|$ is convex, then from (\ref{eq.08}) and (\ref{eq.12}) we obtain that
\begin{align}\label{eq.13}
&\bigg|R\int\int_{\sigma(C,R)}f(x,y,z)d\sigma-3\int\int\int_{\bar{B}(C,R)}f(x,y,z)dv\bigg|\leq\\
&\frac{R^2}{20}\int\int_{\sigma(C,R)}\big|\frac{\partial f}{\partial \rho}\big|(x,y,z)d\sigma+\frac{R^2}{5}\int\int_{\sigma(C,R)}\big|\frac{\partial f}{\partial \rho}\big|(x,y,z)d\sigma=\notag\\
&\frac{R^2}{4}\int\int_{\sigma(C,R)}\big|\frac{\partial f}{\partial \rho}\big|(x,y,z)d\sigma\notag.
\end{align}
Dividing (\ref{eq.13}) with "$4\pi R^3$" we obtain the desired result (\ref{eq.08'}). For the sharpness of (\ref{eq.08'}) consider the function $f:\bar{B}(C,R)\to \mathbb{R}$ defined as
$$f(x,y,z)=R-\sqrt{(x-a)^2+(y-b)^2+(z-c)^2}.$$
By the use of spherical coordinates we have $f(\rho,\varphi,\theta)=R-\rho$, for $\rho\in [0,R]$, $\varphi\in [0,\pi]$ and $\theta\in [0,2\pi]$. With some calculations we obtain that
\begin{align}\label{eq.14}
\frac{1}{\frac{4}{3}\pi R^3}\int\int\int_{\bar{B}(C,R)}f(x,y,z)dv=\frac{1}{\frac{4}{3}\pi R^3}\int_0^{2\pi}\int_0^{\pi}\int_0^R(R-\rho)\rho^2 sin\varphi d\rho d\varphi d\theta=\frac{R}{4}.
\end{align}
Also
\begin{align}\label{eq.15}
&\int\int_{\sigma(C,R)}f(x,y,z)d\sigma=\\
&\int_0^{2\pi}\int_0^{\pi} f(a+R cos\theta sin\varphi, b+R sin\theta cos\varphi, c+R cos\varphi)R^2 sin\varphi d\varphi d\theta=\notag\\
&\int_0^{2\pi}\int_0^{\pi} (R-R) R^2 sin\varphi d\varphi d\theta=0.\notag
\end{align}
On the other hand since $\big|\frac{\partial f}{\partial \rho}\big|=1$, then
\begin{align}
&\frac{1}{16\pi R}\int\int_{\sigma(C,R)}\big|\frac{\partial f}{\partial \rho}\big|(x,y,z)d\sigma=\frac{4\pi R^2}{16\pi R}=\frac{R}{4}\notag,
\end{align}
which along with (\ref{eq.14}) and (\ref{eq.15}) show that (\ref{eq.08'}) is sharp.
\end{proof}
The following is trapezoid type inequalities related to the (\ref{eq.08}) for the mappings defined on $\bar{B}(C,R)$.
\begin{theorem}
Suppose that $\bar{B}(C,R)\subset I^\circ$, where $I\subset R^3$ and consider $f:\bar{B}(C,R)\to \mathbb{R}$ which has continuous partial derivatives with respect to the variables $\rho$, $\varphi$ and $\theta$ on $I^\circ$ in spherical coordinates. If $|\frac{\partial f}{\partial \rho}|$ is convex on $\bar{B}(C,R)$, then
\begin{align}
&\bigg|\frac{1}{\frac{4}{3}\pi R^3}\int\int\int_{\bar{B}(C,R)}f(x,y,z)dv-f(a,b,c)\bigg|\leq\frac{5}{16\pi R}\int\int_{\sigma(C,R)}\big|\frac{\partial f}{\partial \rho}\big|(x,y,z)d\sigma.\notag
\end{align}
\end{theorem}
\begin{proof}
For fixed $\varphi\in [0,\pi]$ and $\theta\in [0,2\pi]$ we have
\begin{align}\label{eq.17}
&\int_0^R\frac{\partial f}{\partial \rho}(a+\rho cos\theta sin\varphi, b+\rho sin\theta cos\varphi, c+\rho cos\varphi) sin\varphi d\rho=\\
&f(a+R cos\theta sin\varphi, b+R sin\theta cos\varphi, c+R cos\varphi)sin\varphi-f(a, b, c) sin\varphi.\notag
\end{align}
Integration with respect to the variables $\varphi\in [0,\pi]$ and $\theta\in [0,2\pi]$ in (\ref{eq.17}) implies that
\begin{align}
&\int_0^{2\pi}\int_0^{\pi}\int_0^R \frac{\partial f}{\partial \rho}(a+\rho cos\theta sin\varphi, b+\rho sin\theta cos\varphi, c+\rho cos\varphi) sin\varphi d\rho d\varphi d\theta=\notag\\
&\int_0^{2\pi}\int_0^{\pi} f(a+Rcos\theta sin\varphi, b+R sin\theta cos\varphi, c+R cos\varphi\Big) sin\varphi d\varphi d\theta-\notag\\
&\int_0^{2\pi}\int_0^{\pi}f(a, b, c) sin\varphi d\varphi d\theta=\frac{1}{R^2}\int\int_{\sigma(C,R)}\big|\frac{\partial f}{\partial \rho}\big|(x,y,z)d\sigma-4\pi f(a,b,c).\notag
\end{align}
So from the convexity of $\big|\frac{\partial f}{\partial \rho}\big|$ we get
\begin{align}\label{eq.19}
&\bigg|\frac{1}{4\pi R^2}\int\int_{\sigma(C,R)}f(x,y,z)d\sigma-f(a,b,c)\bigg|\leq\\
&\frac{1}{4\pi}\int_0^{2\pi}\int_0^{\pi}\int_0^R \big|\frac{\partial f}{\partial \rho}\big|(a+\rho cos\theta sin\varphi, b+\rho sin\theta cos\varphi, c+\rho cos\varphi) sin\varphi d\rho d\varphi d\theta\leq\notag\\
&\frac{1}{4\pi}\int_0^{2\pi}\int_0^{\pi}\int_0^R (1-\frac{\rho}{R})\big|\frac{\partial f}{\partial \rho}\big|(a, b, c) sin\varphi d\rho d\varphi d\theta+\notag\\
&\frac{1}{4\pi}\int_0^{2\pi}\int_0^{\pi}\int_0^R \frac{\rho}{R}\big|\frac{\partial f}{\partial \rho}\big|(a+R cos\theta sin\varphi, b+R sin\theta cos\varphi, c+R cos\varphi) sin\varphi d\rho d\varphi d\theta=\notag\\
&\frac{R}{2}\big|\frac{\partial f}{\partial \rho}\big|(a, b, c)+\frac{1}{8\pi R}\int\int_{\sigma(C,R)}f(x,y,z)d\sigma\notag.
\end{align}
It follows from triangle inequality, (\ref{eq.19}), (\ref{eq.08}) and (\ref{eq.08'}) that
\begin{align*}
&\bigg|\frac{1}{\frac{4}{3}\pi R^3}\int\int\int_{\bar{B}(C,R)}f(x,y,z)dv-f(a,b,c)\bigg|\leq\\
&\bigg|\frac{1}{\frac{4}{3}\pi R^3}\int\int\int_{\bar{B}(C,R)}f(x,y,z)dv-\frac{1}{4\pi R^2}\int\int_{\sigma(C,R)}f(x,y,z)d\sigma \bigg|+\\
&\bigg|\frac{1}{4\pi R^2}\int\int_{\sigma(C,R)}f(x,y,z)d\sigma-f(a,b,c)\bigg|\leq\\
&\frac{1}{16\pi R}\int\int_{\sigma(C,R)}\big|\frac{\partial f}{\partial \rho}\big|(x,y,z)d\sigma+\frac{R}{2}\big|\frac{\partial f}{\partial \rho}\big|(a, b, c)+\frac{1}{8\pi R}\int\int_{\sigma(C,R)}f(x,y,z)d\sigma\leq\notag\\
&\frac{1}{16\pi R}\int\int_{\sigma(C,R)}\big|\frac{\partial f}{\partial \rho}\big|(x,y,z)d\sigma+\frac{1}{8\pi R}\int\int_{\sigma(C,R)}\big|\frac{\partial f}{\partial \rho}\big|(x,y,z)d\sigma+\notag\\
&\frac{1}{8\pi R}\int\int_{\sigma(C,R)}\big|\frac{\partial f}{\partial \rho}\big|(x,y,z)d\sigma=\frac{5}{16\pi R}\int\int_{\sigma(C,R)}\big|\frac{\partial f}{\partial \rho}\big|(x,y,z)d\sigma,\notag
\end{align*}
which implies the desired result.
\end{proof}


\end{document}